\documentclass[11pt,oneside]{amsart}
\usepackage{amsmath,amsfonts,amsthm,amssymb}
\usepackage{MnSymbol}
\usepackage{tikz}
\usetikzlibrary{arrows,decorations.pathmorphing,backgrounds,positioning,fit,calc}
\usepackage{setspace}
\usepackage{fullpage}
\usepackage{hyperref}
\usepackage[subrefformat=parens, labelfont=up]{subcaption}

\newtheorem{thm}{Theorem}

\newtheorem{innercustomthm}{Theorem}
\newenvironment{rethm}[1]
  {\renewcommand\theinnercustomthm{#1}\innercustomthm}
  {\endinnercustomthm}

\newtheorem{lemma}[thm]{Lemma}

\newtheorem{prop}[thm]{Proposition}

\theoremstyle{definition}
\theoremstyle{remark}

\newenvironment{ex}{\refstepcounter{thm}\begin{proof}[Example \emph{\thethm}]}{\end{proof}}
\newenvironment{rem}{\refstepcounter{thm}\begin{proof}[Remark \emph{\thethm}]}{\end{proof}}
\newenvironment{defn}{\refstepcounter{thm}\begin{proof}[Definition \emph{\thethm}]}{\end{proof}}

\newenvironment{warn}{\refstepcounter{thm}\begin{proof}[Warning \emph{\thethm}]}{\end{proof}}
\numberwithin{equation}{section}

\colorlet{dark purple}{red!35!blue}
\colorlet{dark green}{green!70!black}
\colorlet{dark red}{red!80!black}
\colorlet{dark blue}{blue!80!black!80!cyan}

\tikzstyle{mutable}=[inner sep=0.5mm,circle,draw,minimum size=2mm]
\tikzstyle{frozen}=[inner sep=.9mm,rectangle,draw]
\tikzstyle{dot} = [fill=black!25,inner sep=0.5mm,circle,draw,minimum size=1mm]
\tikzstyle{marked}=[inner sep=0.5mm,circle,draw,blue!75!black,fill=blue!50]

\tikzstyle{projection line}=[dark red, line width=.02]

\title{Separating dots with circles}  

\author{James Beyer}
\address{Department of Mathematics \\ University of Oklahoma \\ Norman, OK 73019 \\ USA}
\email{jebeyerou{\char'100}gmail.com}
\author{Jaewon Min}
\address{Department of Mathematics \\ University of Illinois Urbana-Champaign \\ Urbana, IL 61801 \\ USA}
\email{jaewonm2{\char'100}illinois.edu}
\author{Greg Muller}
\address{Department of Mathematics \\ University of Oklahoma \\ Norman, OK 73019 \\ USA}
\email{gmuller{\char'100}ou.edu}


\begin{document}

\begin{abstract}
Given a finite set of points in general position in the plane or sphere, we count the number of ways to separate those points using two types of circles: circles through three of the points, and circles through none of the points (up to an equivalence). In each case, we show the number of circles which separate the points into subsets of size $k$ and $\ell$ is independent of the configuration of points, and we provide an explicit formula in each case.

We also consider how the circles change as the configuration of dots varies continuously. We show that an associated \emph{higher order Voronoi decomposition} of the sphere changes by a sequence of local `moves'. As a consequence, an associated cluster algebra is independent of the configuration of dots, and only depends on the number of dots and the order of the Voronoi decomposition.

%
%
\end{abstract}




\maketitle 




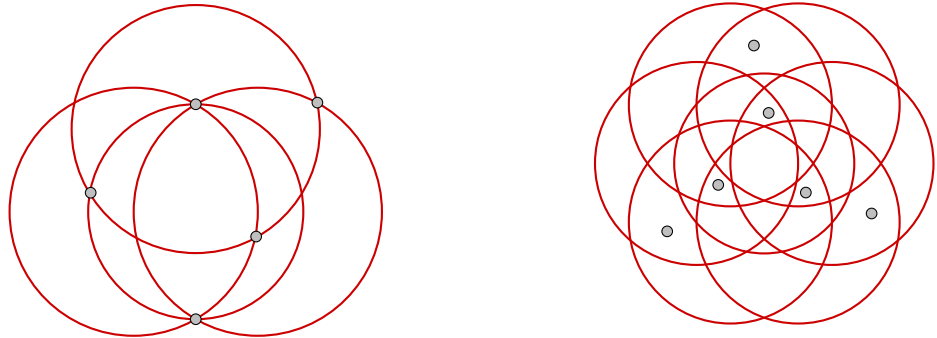
\begin{figure}[h!tb]
\captionsetup[subfigure]{justification=centering}
\centering
\begin{subfigure}[c]{0.45\linewidth}
\centering
    \begin{tikzpicture}[scale=1.1]
         \path[use as bounding box] (-2,-2) rectangle (2,1.3);
         \draw[thick,dark red] (-.75,-.5) circle (1.5);
         \draw[thick,dark red] (0,-.5) circle (1.3);
         \draw[thick,dark red] (.75,-.5) circle (1.5);
         \draw[thick,dark red] (0,.5) circle (1.5);
         \node[dot] at (0,.8) {};
         \node[dot] at (0,-1.8) {};
         \node[dot] at (1.47,.82) {};
         \node[dot] at (-1.27,-.27) {};
         \node[dot] at (.73,-.8) {};
    \end{tikzpicture}
\subcaption{Five dots in general position, and the four incident circles with one dot on either side}
\label{fig: incidentexample}
\end{subfigure}
\begin{subfigure}[c]{0.45\linewidth}
\centering
    \[
    \begin{tikzpicture}[scale=.9]
         \path[use as bounding box] (-2,-2.4) rectangle (2,1.3);
        \draw[thick,dark red] (60:1) circle (1.5);
        \draw[thick,dark red] (120:1) circle (1.5);
        \draw[thick,dark red] (180:1) circle (1.5);
        \draw[thick,dark red] (240:1) circle (1.5);
        \draw[thick,dark red] (300:1) circle (1.5);
        \draw[thick,dark red] (0:1) circle (1.5);
        \draw[thick,dark red] (0,0) circle (1.33);
        \node[dot] at (90+5:1.75) {};
        \node[dot] at (210+5:1.75) {};
        \node[dot] at (330+5:1.75) {};
        \node[dot] at (30-5:-.75) {};
        \node[dot] at (150-5:-.75) {};
        \node[dot] at (270-5:-.75) {};
    \end{tikzpicture}
    \]
\subcaption{Six dots in general position, and seven avoidant circles with three dots on either side}
\label{fig: avoidantexample}
\end{subfigure}
\caption{Incident and avoidant circles for two configurations of dots}
\label{fig: examples}
\end{figure}


\section{Summary of results}

Consider a distinguished set of finitely many points in a plane or sphere, which we will call \textbf{dots}.
We say these dots are in \textbf{general position} if no four dots lie on a circle and (in the case of the plane) no three dots lie on a line.
By these assumptions, any three dots lie on a unique circle, which we call an \textbf{incident circle}, and this incident circle separates the remaining dots into two sets.


Remarkably, if we count incident circles which have fixed numbers of dots on the two sides, the number does not depend on the configuration of the dots and is given by the following formula.

\begin{rethm}{A}\label{thm: incident}
Given at least three dots in general position in the plane or sphere, the number of circles through three of the dots which separate the remaining dots into subsets of size $k$ and $\ell$ is
    \[
    \left\{
    \begin{array}{cc}
    2(k+1)(\ell+1) &\text{if $k\neq \ell$} \\
    (k+1)^2 &\text{if $k= \ell$} \\
    \end{array}
    \right\}
    \]
\end{rethm}

\begin{rem}
Note that we do not distinguish between the two sides of the circle, even in the plane. If we did, the analogous counts may depend on the configuration of the dots; see Example \ref{ex: interiorcircles}.
\end{rem}


We also consider circles which pass through none of the dots, which we call \textbf{avoidant circles}. Of course, there are infinitely many such circles, but they only induce finitely many partitions of the dots. 
If we count such partitions with a fixed number of dots in the two parts, the number again does not depend on the configuration of the dots and is given by the following formula. 

\begin{rethm}{B}\label{thm: avoidant}
Given at least three dots in general position in the plane or sphere, the number of partitions of the dots into non-empty subsets of size $k$ and $\ell$ which can be separated by a circle is
    \[
    \left\{
    \begin{array}{cc}
    2k \ell - k - \ell + 2 &\text{if $k\neq \ell$} \\
    k^2 - k + 1 &\text{if $k= \ell$} \\
    \end{array}
    \right\}
    \]
\end{rethm}

\noindent If we say two avoidant circles are \emph{equivalent} if they induce the same partition of the dots, then this theorem can be interpreted as counting avoidant circles up to equivalence.



Our primary tool in the latter theorem is the \emph{$k$th order Voronoi decomposition} of the sphere, which is the partition into regions based on which $k$-element subset of the dots is closest.
When $k=1$, this is the classical \emph{Voronoi decomposition} into regions based on which dot is the closest.

Regions in the $k$th Voronoi decomposition of the sphere are in bijection\footnote{Except when there are exactly $2k$-many dots, in which case the correspondence is 2-to-1.} with avoidant circles that have $k$-many dots on one side (up to equivalence). Theorem \ref{thm: avoidant} then reduces to a count of these regions, which may be deduced from Theorem \ref{thm: incident} and an Euler characteristic argument.

\begin{rem}
The analogous argument does not work in the plane, since the $k$th order Voronoi regions in the plane correspond to equivalence classes of avoidant circles with $k$-many points in their interior, and the number of such circles can depend on the configuration of the dots.
Theorem \ref{thm: avoidant} is still valid in the plane, but we must deduce it from the spherical case via \emph{stereographic projection}.
\end{rem}

We also consider how the $k$th order Voronoi decomposition of the sphere (and thus the incident and avoidant circles) changes as the configuration of dots moves around. If we color the vertices of this decomposition black and white (based on the number of dots on each side of the corresponding incident circle), we can regard the $k$th order Voronoi decomposition as a bicolored graph in the sphere. 
%
%
As long as the deformation of the dots is `general enough', we show this bicolored graph will change by a sequence of the three local `moves' appearing in Figure \ref{fig: postnikovmoves} (Lemma \ref{lemma: move}). 

\begin{figure}[h!t]
\[
\begin{tikzpicture}[scale=.5,baseline=(a.base)]
        \draw[dashed] (0,0) circle (2);
        \clip (0,0) circle (2);
        \node[dot,thick,fill=white] (1) at (180:1) {};
        \node[dot,thick,fill=white] (3) at (0:1) {};

        \draw[thick] (1) to (3);
        \draw[thick] (1) to (-5,5);
        \draw[thick] (1) to (-5,-5);
        \draw[thick] (3) to (5,5);
        \draw[thick] (3) to (5,-5);

        \node[opacity=0] (a) at (0,0) {$a$};
\end{tikzpicture}
\leftrightarrow
\begin{tikzpicture}[scale=.5,baseline=(a.base)]
        \draw[dashed] (0,0) circle (2);
        \clip (0,0) circle (2);
        \node[dot,thick,fill=white] (2) at (90:1) {};
        \node[dot,thick,fill=white] (4) at (270:1) {};

        \draw[thick] (2) to (4);
        \draw[thick] (2) to (-5,5);
        \draw[thick] (4) to (-5,-5);
        \draw[thick] (2) to (5,5);
        \draw[thick] (4) to (5,-5);

        \node[opacity=0] (a) at (0,0) {$a$};
\end{tikzpicture}
\hspace{.75cm}
\begin{tikzpicture}[scale=.5,baseline=(a.base)]
        \draw[dashed] (0,0) circle (2);
        \clip (0,0) circle (2);
        \node[dot,thick,fill=black] (1) at (180:1) {};
        \node[dot,thick,fill=white] (2) at (90:1) {};
        \node[dot,thick,fill=black] (3) at (0:1) {};
        \node[dot,thick,fill=white] (4) at (270:1) {};

        \draw[thick] (1) to (2);
        \draw[thick] (2) to (3);
        \draw[thick] (3) to (4);
        \draw[thick] (4) to (1);
        \draw[thick] (1) to (-5,0);
        \draw[thick] (2) to (0,5);
        \draw[thick] (3) to (5,0);
        \draw[thick] (4) to (0,-5);

        \node[opacity=0] (a) at (0,0) {$a$};
\end{tikzpicture}
\leftrightarrow
\begin{tikzpicture}[scale=.5,baseline=(a.base)]
        \draw[dashed] (0,0) circle (2);
        \clip (0,0) circle (2);
        \node[dot,thick,fill=white] (1) at (180:1) {};
        \node[dot,thick,fill=black] (2) at (90:1) {};
        \node[dot,thick,fill=white] (3) at (0:1) {};
        \node[dot,thick,fill=black] (4) at (270:1) {};

        \draw[thick] (1) to (2);
        \draw[thick] (2) to (3);
        \draw[thick] (3) to (4);
        \draw[thick] (4) to (1);
        \draw[thick] (1) to (-5,0);
        \draw[thick] (2) to (0,5);
        \draw[thick] (3) to (5,0);
        \draw[thick] (4) to (0,-5);

        \node[opacity=0] (a) at (0,0) {$a$};
\end{tikzpicture}
\hspace{.75cm}
\begin{tikzpicture}[scale=.5,baseline=(a.base)]
        \draw[dashed] (0,0) circle (2);
        \clip (0,0) circle (2);
        \node[dot,thick,fill=black] (2) at (90:1) {};
       \node[dot,thick,fill=black] (4) at (270:1) {};

        \draw[thick] (2) to (4);
        \draw[thick] (2) to (-5,5);
        \draw[thick] (4) to (-5,-5);
        \draw[thick] (2) to (5,5);
        \draw[thick] (4) to (5,-5);

        \node[opacity=0] (a) at (0,0) {$a$};
\end{tikzpicture}
\leftrightarrow
\begin{tikzpicture}[scale=.5,baseline=(a.base)]
        \draw[dashed] (0,0) circle (2);
        \clip (0,0) circle (2);
        \node[dot,thick,fill=black] (1) at (180:1) {};
        \node[dot,thick,fill=black] (3) at (0:1) {};

        \draw[thick] (1) to (3);
        \draw[thick] (1) to (-5,5);
        \draw[thick] (1) to (-5,-5);
        \draw[thick] (3) to (5,5);
        \draw[thick] (3) to (5,-5);

        \node[opacity=0] (a) at (0,0) {$a$};
\end{tikzpicture}
\]
\caption{Three local moves between bicolored graphs in the sphere}
\label{fig: postnikovmoves}
\end{figure}
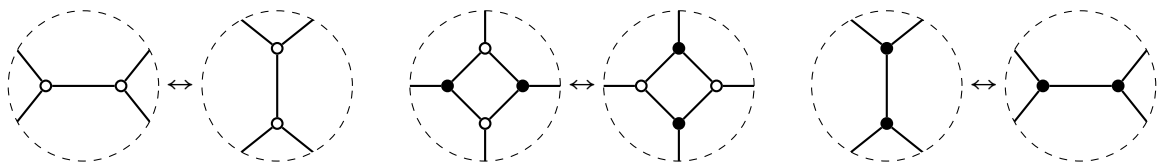

Since any two configurations of the same number of dots can be connected by such a deformation, we deduce the following.





\begin{rethm}{C}\label{thm: move equiv}
Let $0<k<n>3$. The $k$th order Voronoi decompositions of any two sets of $n$-many dots in general position in the sphere may be related by a sequence of the local moves in Figure \ref{fig: postnikovmoves}.
\end{rethm}

These same moves appeared in Postnikov's seminal work on planar bicolored graphs \cite{Pos}. 
Graphs related by these moves have many features in common, for example, they determine isomorphic cluster algebras. 
%
The theorem implies that the $k$th order Voronoi decompositions of any two configurations of $n$-many dots in the sphere determine isomorphic cluster algebras. Cluster algebras constructed this way include several important examples, such as the \emph{Markov cluster algebra} and the \emph{$X_7$ cluster algebra}; preliminary results are discussed briefly in Section \ref{sec: future}.

\begin{rem}
Versions of these counting problems and the resulting formulas have been considered in many previous works, such as the $k=\ell$ version of Theorem \ref{thm: incident} appearing in \cite{Ard04}. Most notably, Theorems \ref{thm: incident} and \ref{thm: avoidant} were proven in full generality in \cite{CHP22}. In that paper, the counts are deduced by constructing spherical Voronoi decompositions from gluing together pairs of planar Voronoi decompositions and then using a symmetry in planar Voronoi decompositions observed in \cite{Lin03} (see Appendix \ref{appendix}).
Our argument suggests that the more natural derivation may be the other direction, by first counting the numbers of cells in spherical higher order Voronoi decompositions and then using the gluing map in \cite{CHP22} to deduce the symmetry in \cite{Lin03}.
\end{rem}



\subsection*{Acknowledgements}

The authors are most grateful for the organizers of the Early Career Conference in Combinatorics held at the University of Illinois Urbana-Champaign in June 2024, where a version of this problem was presented in an open problem session, leading to this collaboration.

\section{Dots in the sphere or plane}

Throughout this note, we let $D$ denote a finite set of points, which we call \textbf{dots}, which may be in either a plane $P$ or a sphere $S$. We say the configuration of dots $D$ is in \textbf{general position} if
\begin{itemize}
    \item no four dots lie on a circle, and
    \item (in the case of the plane) no three dots lie on a line.
\end{itemize}

We can translate between configurations of dots in the plane and in the sphere via a \emph{stereographic projection} map.
To define such a map, choose an embedding of both the sphere $S$ and the plane $P$ in 3-dimensional space, and choose a {point} $p_\infty \in S$ whose tangent plane is parallel to $P$ but not equal to $P$.\footnote{For example, one may always choose $p_\infty$ to be the furthest point on $S$ from $P$.}
The corresponding \textbf{stereographic projection} of a point $p$ in $S\smallsetminus p_\infty$ is defined to be the unique point $\pi(p)$ on $P$ which lies on the line connecting $p$ and $p_\infty$.




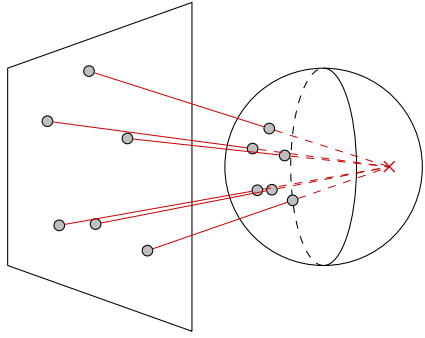
\begin{figure}[h!t]
    \begin{tikzpicture}[scale=1.75,rotate=90]
        \node[dot] (d1) at (-.445,.258) {};
        \node[dot] (d2) at (.347,.348) {};
        \node[dot] (d3) at (.728,.032) {};
        \node[dot] (d4) at (.216,-.26) {};
        \node[dot] (d5) at (-.435,-.018) {};
        \node[dot] (d6) at (-.637,-.412) {};

        \draw (-1.25,-.75) to (1.25,-.75) to (.75,.65) to (-.75,.65) to (-1.25,-.75);

        \begin{scope}[yshift=-1.75cm]
            \draw (0,0) circle (.75);
    
            \draw (-.75,0) arc [start angle=180, end angle=360, x radius=.75, y radius=.25];
            \draw[dashed] (-.75,0) arc [start angle=180, end angle=0, x radius=.75, y radius=.25];
    
            \coordinate (point) at (0,-.5);
            \node[dark red] at (point) {$\times$};
        \end{scope}

        \def\ta{.4};
        \def\tb{.4};
        \def\tc{.4};
        \def\td{.4};
        \def\te{.4};
        \def\tf{.4};

        \node[dot] (p1) at ($\ta*(d1)+{1-\ta}*(point)$) {};
        \node[dot] (p2) at ($\tb*(d2)+{1-\tb}*(point)$) {};
        \node[dot] (p3) at ($\tc*(d3)+{1-\tc}*(point)$) {};
        \node[dot] (p4) at ($\td*(d4)+{1-\td}*(point)$) {};
        \node[dot] (p5) at ($\te*(d5)+{1-\te}*(point)$) {};
        \node[dot] (p6) at ($\tf*(d6)+{1-\tf}*(point)$) {};

        \draw[projection line] (d1) to (p1);
        \draw[projection line] (d2) to (p2);
        \draw[projection line] (d3) to (p3);
        \draw[projection line] (d4) to (p4);
        \draw[projection line] (d5) to (p5);
        \draw[projection line] (d6) to (p6);
        \draw[projection line,dashed] (p1) to (point);
        \draw[projection line,dashed] (p2) to (point);
        \draw[projection line,dashed] (p3) to (point);
        \draw[projection line,dashed] (p4) to (point);
        \draw[projection line,dashed] (p5) to (point);
        \draw[projection line,dashed] (p6) to (point);

    \end{tikzpicture}
\caption{Stereographic projection of 6 dots}
\label{fig: stereo}
\end{figure}

Stereographic projection defines a homeomorphism $\pi$ between the punctured sphere $S\smallsetminus p_\infty$ and the plane $P$ which has many nice properties; e.g.~stereographic projection induces bijections
\[
\begin{aligned}
\{\text{circles in $S$ which contain $p_\infty$}\} &\xrightarrow{\sim} \{\text{lines in $P$}\}
\\
\{\text{circles in $S$ which do not contain $p_\infty$}\} &\xrightarrow{\sim} \{\text{circles in $P$}\}
\end{aligned}
\]
As an immediate consequence, configurations of dots in general position in one space can be translated to configurations of dots in general position in the other space.
\begin{prop}
A configuration of dots $D$ in the sphere $S$ is in general position if and only if there is a stereographic projection $\pi:S\smallsetminus p_\infty\rightarrow P$ such that $\pi(D)$ is in general position in the plane $P$.
\end{prop}

As we will see, the numbers of incident and avoidant circles we want to count are also preserved by stereographic projection.
As a consequence, we can work primarily in one surface (usually, the sphere $S$) but deduce results in both surfaces.


\begin{warn}
Even when $\pi$ sends a circle to a circle, it does not preserve the centers of those circles.
Similarly, it does not send geodesics in $S$ (i.e.~great circles) to geodesics in $P$ (i.e.~lines). This leads to a qualitative difference between the \emph{higher order Voronoi decompositions} (see Appendix \ref{appendix}).
\end{warn}

\section{Counting incident circles}

Let $D$ be a finite set of dots in general position in the plane $P$ or sphere $S$. 
Any three dots in $D$ lie on a unique \textbf{incident circle},\footnote{
The circle through three points in the plane is sometimes called the \emph{circumcircle}, particularly when the three points are the vertices of a triangle. We avoid this language, as we are not interested in the triangle spanned by the dots.}
and this circle separates the remaining dots into two sets. 
%
%
%
We are interested in counting incident circles 
with fixed numbers of dots on either side. We start with the following, which will form the base case of our induction.

\begin{prop}\label{prop: incident0}
Let $n\geq4$. Given $n$-many dots $D$ in general position in the sphere $S$, there are $(2n-4)$-many incident circles with no dots on one side.
\end{prop}

\begin{proof}
Choose an embedding of the sphere $S$ in $\mathbb{R}^3$, and let $H$ denote the convex hull of $D$ as a subset of $\mathbb{R}^3$. 
Since the dots are in general position, at most 3 dots can lie on any plane in $\mathbb{R}^3$, and so each face of the convex hull $H$ must be a triangle. Since each edge of $H$ lies in two faces and each triangle has three edges,
\[ 
2(\text{$\#$ of edges in $H$})
=
3(\text{$\#$ of faces in $H$})
\]
The vertices of $H$ are the $n$-many dots $D$, so the Euler characteristic of $H$ tells us that
\[
2 = n - (\text{$\#$ of edges in $H$}) + (\text{$\#$ of faces in $H$}) 
\]
Solving the previous two equations tells us that the number of faces of $H$ must be $2n-4$.

Given three dots $\{a,b,c\}\subset D$, the incident circle through those dots is the intersection of the affine plane through $\{a,b,c\}$ in $\mathbb{R}^3$ with the sphere $S$. Therefore, the incident circle through $\{a,b,c\}$ has no dots on one side if and only if the plane through $\{a,b,c\}$ bounds a half-space in $\mathbb{R}^3$ which contains the convex hull $H$; equivalently, $\{a,b,c\}$ lie on a face of the convex hull $H$. Since each face of $H$ is a triangle, the face uniquely determines the three dots $\{a,b,c\}$. Therefore, we have a bijection between incident circles with no dots on one side and faces in the convex hull $H$.
\end{proof}



For the next proposition, we want to distinguish between the two sides of an incident circle. We can do this by choosing an orientation, so  we can distinguish the \textbf{left side} and the \textbf{right side}.

\begin{prop}\label{prop: orientedincident}
Let $n\geq3$ and $0\leq k \leq n-3$. Given $n$-many dots $D$ in general position in the sphere $S$, there are 
\[I_{k,n} := 2(k+1)(n-k-2)\]
oriented incident circles with $k$-many dots on their left side (and thus $n-k-3$ dots on the right).
\end{prop}

\begin{proof}
If $n=3$, then $k=0$. Since the unique incident circle through the three dots has $2=I_{0,3}$ orientations, the theorem holds in this case. 

We prove the $n\geq4$ case by induction on $k$. 
If $k=0$, then Proposition \ref{prop: incident0} implies there are $(2n-4)$-many incident circles with no dots on one side. Since $n\geq4$, the other side must have at least 1 dot and there is a unique orientation of each such circle which has no dots on the left side. Thus, there are $I_{0,n}=2\cdot 1\cdot(n-2)$-many oriented incident circles with $0$ dots on their left side.

Next, assume the proposition holds for $k-1$, and consider a configuration $D$ of $n$-many dots in general position in $S$. 
There are $n$-many ways to choose a dot $d\in D$ to delete; by assumption, there are then $I_{k-1,n-1}$-many ways to choose an oriented incident circle through three of the remaining $(n-1)$-many dots which has $(k-1)$-many dots on its left side. 
Therefore, there are
\[
n \cdot I_{k-1,n-1} = 2nk(n-k-2)
\]
ways to delete a dot and then choose an oriented incident circle with $k-1$ dots on its left side.

Alternatively, we could first choose the circle, then choose the dot to delete. This can be done in two disjoint ways.
    \begin{itemize}
        \item First choose an oriented incident circle with $(k-1)$-many dots on its left side, then choose a dot on that circle's right side to delete. By assumption, the number of ways to do this is
        \[
        I_{k-1,n}\cdot (n-k-2) = 2k(n-k-1)(n-k-2)
        \]
        \item First choose an oriented incident circle with $k$-many dots on its left side, then choose a dot on that circle's left side to delete. The number of ways to do this is
        \[
        (\text{$\#$ of oriented incident circles of $D$ with $k$-many dots on their left side})\cdot k
        \]
    \end{itemize}
Since this constructs the same set as before, the sum of these two counts is equal to $nI_{k-1,n-1}$. Solving this equality for the desired number of circles,
\[
(\text{$\#$ of oriented incident circles of $D$ with $k$-many dots on their left side})
=2(k+1)(n-k-2)
\]
This completes the induction.
\end{proof}

\begin{rem}
The heart of the proof is a double counting argument which can be summarized as
\[ nI_{k-1,n-1} = (n-k-2)I_{k-1,n} + kI_{k,n} 
\qedhere
\]
\end{rem}


It is now a simple matter to adapt this to count unoriented incident circles.

\begin{rethm}{\ref{thm: incident}}
Given at least three dots in general position in the plane or sphere, the number of circles through three of the dots which separate the remaining dots into subsets of size $k$ and $\ell$ is
    \[
    \left\{
    \begin{array}{cc}
    2(k+1)(\ell+1) &\text{if $k\neq \ell$} \\
    (k+1)^2 &\text{if $k= \ell$} \\
    \end{array}
    \right\}
    \]
\end{rethm}

\begin{proof}
Since exactly 3 dots lie on an incident circle, there must be $n:=k+\ell+3$ dots. 

If $k\neq \ell$, then such an incident circle has a unique orientation with $k$-many dots on its left side. By Proposition \ref{prop: orientedincident}, the number of such oriented incident circles in the sphere is 
\[
2(k+1)(n-k-2) = 2(k+1)(\ell+1)
\]

If $k=\ell$, then then such an incident circle has two orientations with $k$-many dots on the left side. By Proposition \ref{prop: orientedincident}, the number of such oriented incident circles in the sphere is 
\[
2(k+1)(n-k-2) = 2(k+1)(\ell+1)= 2(k+1)^2
\]
Dividing by $2$, the number of unoriented incident circles with $k$-many dots on one side is $(k+1)^2$.

This proves the theorem for dots in the sphere $S$.
Since incident circles and the sets of dots on either side are preserved by stereographic projection, the theorem holds for dots in the plane $P$.
\end{proof}

\begin{ex}\label{ex: interiorcircles}
Given a configuration of dots in the plane, it may seem more natural to count incident circles with a fixed number of dots on the inside (rather than on either of the sides). However, the number of such circles can depend on the specific configuration.

To see this, consider the two configurations of 4 dots in general position in the plane given in Figure \ref{fig: fourdots}. Each configuration admits $\binom{4}{3}=4$ incident circles. In Figure \ref{fig: fourdotsonecircle}, only one of those circles has a dot on the inside, while, in Figure \ref{fig: fourdotstwocircles}, two of those circles have a dot on the inside.

\begin{figure}[h!tb]
\captionsetup[subfigure]{justification=centering}
\centering
\begin{subfigure}[c]{0.45\linewidth}
\centering
\begin{tikzpicture}[inner sep=0.5mm,scale=1.25,auto]
	\draw[dark red] (0,0) circle (1);
	\node[dot] (0) at (0,0) [circle,draw] {};
	\node[dot] (1) at (90:1) [circle,draw] {};
	\node[dot] (2) at (-30:1) [circle,draw] {};
	\node[dot] (3) at (210:1) [circle,draw] {};
\end{tikzpicture}
\subcaption{A configuration of four dots admitting a unique incident circle with one interior dot}
\label{fig: fourdotsonecircle}
\end{subfigure}
\begin{subfigure}[c]{0.45\linewidth}
\centering
\begin{tikzpicture}[inner sep=0.5mm,scale=1.25,auto]
	\draw[dark red] (.5,0) circle (1);
	\draw[dark red] (-.5,0) circle (1);
	\node[dot] (0) at (.5,0) [circle,draw] {};
	\node[dot] (1) at (-.5,0) [circle,draw] {};
	\node[dot] (2) at (0,.866) [circle,draw] {};
	\node[dot] (3) at (0,-.866) [circle,draw] {};
\end{tikzpicture}
\subcaption{A configuration of four dots admitting two incident circles with one interior dot}
\label{fig: fourdotstwocircles}
\end{subfigure}
\caption{Two configurations of four dots in general position in the plane}
\label{fig: fourdots}
\end{figure}
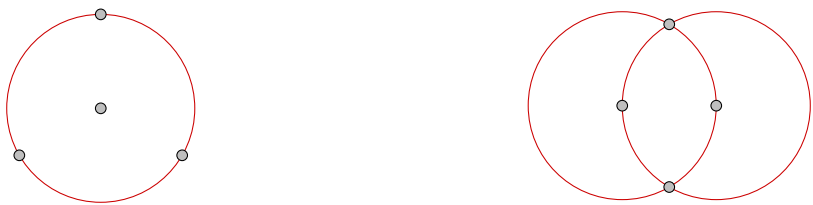

Note that the three incident circles not pictured in 
Figure \ref{fig: fourdotsonecircle} and the two incident circles not pictured in Figure \ref{fig: fourdotstwocircles} each have 0 dots in the inside and 1 dot on the outside.
Therefore, in both configurations, all four incident circles separate the remaining dot(s) into subsets of sizes $k=1$ and $\ell=0$, matching the formula in Theorem \ref{thm: incident}.
\end{ex}






\section{Higher order Voronoi decompositions}\label{sec: higher Voronoi}

\def\dis{\operatorname{dist}}

Our primary tool for counting avoidant circles will be the \emph{$k$-th order Voronoi decomposition} of the sphere determined by a configuration of dots $D$. Roughly speaking, this decomposes the sphere into regions based on which $k$-element subset of the dots is the closest to a given point.\footnote{
Higher order Voronoi decompositions may also be defined in the plane (or any metric space), but they are not as nicely behaved as that of the sphere. In particular, the number of vertices, edges, and faces may depend on the configuration. See Appendix \ref{appendix} for a discussion of the planar case and the relation to the spherical case.
}  A survey of results relating to higher order Voronoi diagrams can be found in Edelsbrunner’s book on algorithms in combinatorial geometry \cite{EdelsbrunnerBook}, though our discussion here assumes no prior knowledge.








Given a configuration of dots $D$ in the sphere $S$ and an arbitrary point $p\in S$, 
index the dots $d_1,d_2,\dotsc,d_n$ in $D$ by weakly increasing distance from $p$; that is,
\[
\dis(p,d_1)\leq \dis(p,d_2) \leq \dotsb \leq \dis(p,d_n) 
\]
where $\dis(p,q)$ denotes the geodesic distance between any two points $p,q\in S$.

Fix an integer $0<k<n$. For each point $p\in S$, we use the temporary indexing above to define a pair of sets (which ultimately do not depend on the choice of indexing)
\[
\begin{aligned}
D_{-}(p)&:= \{d\in D \mid \dis(p,d) \leq \dis(p,d_k) \} \\
D_{+}(p)&:= \{d\in D \mid \dis(p,d) \geq \dis(p,d_{k+1}) \} \\
\end{aligned}
\]
Informally, $D_{-}(p)$ is the set of dots that are at worst tied for $k$-closest to $p$, and $D_{+}(p)$ is the set of dots that are at best tied for $(k+1)$-closest to $p$.
While these two sets are often disjoint, they will overlap whenever there is a tie between the $k$th closest and the $(k+1)$-closest dots to $p$.

\begin{defn}
Let $0<k<n>3$ and let $D$ be a set of $n$-many dots in the sphere $S$. A \textbf{(closed) $k$th-order Voronoi stratum} is a subset $V\subset S$ of the form 
\[
\{ q\in S \text{ such that $D_-(p)\subseteq D_-(q)$ and $D_+(p)\subseteq D_+(q)$}\}
\]
for some point $p\in S$. The $k$th order Voronoi strata collectively form the \textbf{$k$th order Voronoi decomposition} of $S$ determined by $D$.
\end{defn}


While many different points $p\in S$ may determine the same $k$th order Voronoi stratum, each $k$th order Voronoi stratum $V$ is uniquely determined by the pair of subsets (for any $p$ which defines $V$)
\begin{equation}\label{eq: Voronoisets}
D_-(V) := D_-(p)
\text{ and }
D_+(V) := D_+(p).
\end{equation}

The possible shapes of these strata are given as follows.

\begin{prop}\label{prop: stratumtypes}
Let $0<k<n>3$ and let $D$ be a configuration of $n$-many dots in general position in $S$. Then every $k$th order Voronoi stratum $V$ must be one of the following.
\begin{enumerate}
    \item A spherical polygon.
    \item A great circular arc.
    \item A single point.
\end{enumerate}
The cases are characterized by $D_-(V)\cap D_+(V)$ consisting of $0$, $2$, or $3$ dots, respectively.
\end{prop}

\noindent We refer to these three types of strata as \textbf{regions}, \textbf{edges}, and \textbf{vertices} in the $k$th order Voronoi decomposition, respectively.

\begin{proof}
Letting $D_-(V)$ and $D_+(V)$ be as in \eqref{eq: Voronoisets}, the definition of $V$ can be reformulated as follows.
\[
\begin{aligned}
V 
&= \{ q\in S \text{ such that $D_-(V)\subseteq D_-(q)$ and $D_+(V)\subseteq D_+(q)$}\}
\\
&= \{ q\in S  \text{ such that, for all $d_-\in D_-(V), d_+\in D_+(V)$, $\dis(q,d_-)\leq \dis(q,d_+)$}\}
\\
&= 
\bigcap_{d_-\in D_-(V)}~ \bigcap_{d_+\in D_+(V)}
\{ q\in S  \text{ such that $\dis(q,d_-)\leq \dis(q,d_+)$}\}
\\
\end{aligned}
\]
Each set in the last expression is a closed hemisphere (except when $d_-=d_+$, in which case it is all of $S$), and so $V$ is a finite intersection of closed hemispheres; in particular, it is convex.

If $D_-(V) \cap D_+(V)$ is empty and $p$ is a point which defines $V$, then 
\[ \dis(p,d_k) < \dis(p,d_{k+1}) \]
By continuity, $D_-$ and $D_+$ are constant on an open neighborhood of $p$, and so $V$ contains an open neighborhood of $p$. Since $V$ is a finite intersection of closed hemispheres, $V$ is a spherical polygon.

If $D_-(V) \cap D_+(V) = \{d_i,d_j\}$, then 
    \[
    \dis(p,d_i)=\dis(p,d_j)
    \]
for every point $p\in V$, and so $V$ is contained in the perpendicular bisector between $d_i$ and $d_j$. Since $V$ is convex, $V$ must be a single arc along this perpendicular bisector (which is a great circle).

If $D_-(V)\cap D_+(V)$ contains three distinct dots $d_i,d_j,d_k$, then for every $p \in V$,
    \[
    \dis(p,d_i)=\dis(p,d_j) = \dis(p,d_k) 
    \]
Since $V$ is convex, $V$ must be one of the two centers of the incident circle through $d_i,d_j,d_k$.
\end{proof}

The last step of the proof shows that vertices of the $k$th order Voronoi decomposition occur at centers of incident circles, which we deepen as follows.
As with sides, we can distinguish the two centers of a circle by choosing an orientation of the circle and considering the \textbf{left center}. 


\begin{prop}\label{prop: blackwhite3}
Assume a configuration of dots $D$ is in general position in the sphere.
\begin{enumerate}
    \item The vertices in the $k$th order Voronoi decomposition are the left centers of oriented incident circles with either $(k-2)$ or $(k-1)$-many dots on their left side.
    \item Every vertex in the $k$th order Voronoi decomposition is contained in exactly three edges.
\end{enumerate}

\end{prop}

\begin{proof}
Let $p$ be a vertex in the $k$th order Voronoi decomposition.
Since every dot in $D_-(p)\cap D_+(p)$ lies on the same circle centered at $p$,
general position implies that $|D_-(p)\cap D_+(p)|=3$. Since $D_-(p)$ contains at least $k$-many dots and $D_+(p)$ contains at least $(n-k)$-many dots, then 
$D_-(p) \smallsetminus D_+(p)$ has either $(k-1)$ or $(k-2)$ many dots.
Orient the incident circle through $D_-(p)\cap D_+(p)$ so that $p$ is its left center. Then $D_-(p)\smallsetminus D_+(p)$ consists of dots on its left side; this set must contain either $(k-1)$ or $(k-2)$ many dots.

Let $d_i,d_j,d_k$ denote the dots in $D_-(p)\cap D_+(p)$. Then $p$ lies on the perpendicular bisector between $d_i$ and $d_j$. 
If $|D_-(p)\smallsetminus D_+(p)| = k-2$, then the $k$th order Voronoi decomposition has an edge along this great circle in the direction that makes $d_k$ closer than $d_i$ and $d_j$. If $|D_-(p)\smallsetminus D_+(p)| = k-1$, then the $k$th order Voronoi decomposition has an edge along this great circle in the direction that makes $d_k$ further than $d_i$ and $d_j$. Repeating for the other two subsets of $D_-(p)\cap D_+(p)$ gives the three edges containing $p$.
\end{proof}











Proposition \ref{prop: blackwhite3} lets us distinguish between two flavors of vertex, which we translate into colors. 
\begin{itemize}
    \item A vertex is \textbf{white} if it is the left center of an oriented incident circle with $(k-2)$-many dots on its left side.
    \item A vertex is \textbf{black} if it is the left center of an oriented incident circle with $(k-1)$-many dots on its left side.
\end{itemize}
We can restate this in terms of distance as follows. A point $p$ is a white vertex if there are dots tied for $(k-1)$st, $k$th, and $(k+1)$st closest to $p$; it is a black vertex if there are dots tied for $k$th, $(k+1)$st, and $(k+2)$nd closest to $p$.

With this coloring, the vertices and edges of the $k$th order Voronoi decomposition become a $3$-regular \textbf{bicolored graph} embedded in the sphere (see Figure \ref{fig: sphericalVoronoi}). 
This bicoloring is necessary for the moves considered in Figure \ref{fig: postnikovmoves} and in Section \ref{sec: Relations}.

\begin{figure}[h!!tb]
\includegraphics[height=6cm]{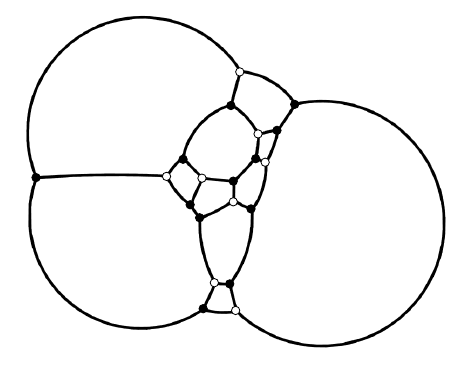}
\caption{The stereographic projection of a $2$nd order Voronoi decomposition of the sphere determined by a configuration of $6$ dots, with bicolored vertices}
\label{fig: sphericalVoronoi}
\end{figure}

Combined with the previous section, Proposition \ref{prop: blackwhite3} lets us count each type of stratum.

\begin{thm}\label{thm: stratacount}
Let $0<k<n > 3$. Given $n$-many dots in general position in the sphere, the $k$th order Voronoi decomposition has
\begin{itemize}
    \item 
    $2(2nk-2k^2-n)$-many vertices (of which $I_{k-2,n}$ are white and $I_{k-1,n}$ are black),
    \item $3(2nk-2k^2-n)$-many edges, and
    \item $(2nk-2k^2-n+2)$-many regions.
\end{itemize}


\end{thm}

\begin{proof}
By Propositions \ref{prop: orientedincident} and \ref{prop: blackwhite3}.1, the number of vertices is
\[
I_{k-1,n} + I_{k-2,n} = 2k(n-k-1) + 2(k-1)(n-k) = 2(2nk-2k^2 -n)
\]
Since every edge contains 2 vertices and every vertex is in 3 edges (Proposition \ref{prop: blackwhite3}.2), the number of edges is $\frac{3}{2}$ times the number of vertices; that is, $3(2nk-2k^2 -n)$.
Finally, since the Euler characteristic of the sphere is $2$, the number of regions is
\[
2 - \text{(\# of vertices}) + \text{(\# of edges)} = 2nk - 2k^2 - n + 2
\qedhere
\]
\end{proof}

\begin{ex}
The stereographic projection of a $2$nd order Voronoi decomposition of the sphere determined by $n=6$ dots is given in Figure \ref{fig: sphericalVoronoi}. It has
$I_{0,6}= 8$ white dots, $I_{1,6} =12$ black dots, $30$ edges, and $12$ regions; this matches the $(k,n)=(2,6)$ case of Theorem \ref{thm: stratacount}.
\end{ex}

Finally, we consider a special case where additional structure appears. Recall that the sphere has a distinguished \textbf{antipodal symmetry}  $\sigma$ which sends every point to its \emph{antipode} (the furthest point on the sphere); this is the unique non-trivial symmetry which sends every great circle to itself.

\begin{prop}\label{prop: antipodal}
Let $n>3$ be even, and let $D$ be a configuration of $n$-many dots in general position in the sphere. Then the $\frac{n}{2}$th order Voronoi decomposition of the sphere is preserved by the antipodal symmetry $\sigma$ (that is, $\sigma$ sends strata to strata), and $\sigma$ swaps the colors of the vertices.
\end{prop}

\begin{proof}
%
The antipodal symmetry reverses the weak ordering of the dots by distance; that is,
\[
\dis(p,d_i) \leq \dis( p,d_j) \Leftrightarrow \dis(\sigma(p) , d_i) \geq \dis(\sigma(p),d_j)
\]
Therefore, $D_-(\sigma(p)) = D_+(p)$ and $D_+(\sigma(p)) = D_-(p)$.

If $p$ is a white vertex, then $|D_-(p)|=\frac{1}{2}n+1$ and $|D_+(p)|=\frac{1}{2}n+2$. Therefore, $|D_-(\sigma(p))|=\frac{1}{2}n+2$ and $|D_+(\sigma(p))|=\frac{1}{2}n+1$ and so $\sigma(p)$ is a black vertex. Similar considerations show that $\sigma$ sends black vertices to white vertices, edges to edges, and regions to regions.
%
\end{proof}

\section{Counting avoidant circles}

We return our attention to counting circles. Given a finite configuration of dots $D$ in the plane or sphere, an \textbf{avoidant circle} is a circle which passes through none of the dots. An avoidant circle induces a partition of the dots $D$ into two parts. 
We say two (unoriented) avoidant circles are \textbf{equivalent} if they induce the same set partition of the dots $D$.

As before, it will be useful to sometimes consider \textbf{oriented avoidant circles}, as this lets us distinguish between the two sides and centers. We say two oriented avoidant circles are \textbf{equivalent} if they have the same set of dots on their left sides (and so the same set of dots on their right sides).




\begin{prop}\label{prop: Voronoivertex}
For any subset $D_-\subset D$ and any point $p$, the following are equivalent.
\begin{enumerate}
    \item $p$ is the left center of an oriented circle with $D_-$ on its left and $D_+:=D\smallsetminus D_-$ on its right.
    \item $\dis(p,d_-) < \dis(p,d_+) $ for every $d_-\in D_-$ and $d_+\in D_+:=D\smallsetminus D_-$.
    \item $p$ is in the interior of a higher order Voronoi region $V$ with $D_-(V)=D_-$ and $D_+(V) = D\smallsetminus D_-$.
\end{enumerate}
\end{prop}

\begin{proof}
If $C$ is the circle consisting of points of distance $r$ from $p$, and $C$ is oriented so that $p$ is on the left side, then the dots on the left side of $C$ are those of distance less than $r$ from $p$. Therefore, there is an oriented avoidant circle with left center at $p$ if and only if the dots in $D_-$ are closer to $p$ than the dots in $D_+$. As was shown in the proof of Proposition \ref{prop: stratumtypes}, the interior of the $k$th Voronoi region $V$ is defined by points $p$ with $D_-(V)=D_-(p)$ and $D_+(V)=D_+(p)$.
\end{proof}





This establishes a bijection between regions in the $k$th order Voronoi regions and equivalence classes of oriented avoidant circles with $k$-many dots on their left. 

\begin{prop}\label{prop: orientedavoidant}
Let $0<k<n>3$. Given $n$-many dots in general position in the sphere or plane, there are 
$
(2nk -2k^2 - n + 2)
$-many oriented avoidant circles (up to equivalence) which have $k$-many dots on their left side.
\end{prop}

\begin{proof}
The previous proposition reduces this to a count of $k$th order Voronoi regions. In the case of the sphere, Theorem \ref{thm: stratacount} shows the number of such regions is $2nk-2k^2-n+2$.
Since equivalence classes of oriented avoidant circles are preserved by stereographic projection, there are the same number in the plane.
\end{proof}

We can then deduce the corresponding count of unoriented avoidant circles.

\begin{rethm}{\ref{thm: avoidant}}
Given at least four dots in general position in the plane or sphere, the number of partitions of the dots into subsets of size $k$ and $\ell$ which can be separated by a circle is
    \[
    \left\{
    \begin{array}{cc}
    2k \ell - k - \ell + 2 &\text{if $k\neq \ell$} \\
    k^2 - k + 1 &\text{if $k= \ell$} \\
    \end{array}
    \right\}
    \]
\end{rethm}

\begin{proof}
Since no dots lie on an avoidant circle, there must be $n:=k+\ell$ dots. 

If $k\neq \ell$, then such an avoidant circle has a unique orientation with $k$-many dots on its left side. By Proposition \ref{prop: orientedavoidant}, the number of such oriented avoidant circles in the sphere is 
\[
2(k+\ell)k -2k^2 - (k+\ell) + 2
=
2k \ell - k - \ell + 2
\]

If $k=\ell$, then then such an avoidant circle has two orientations with $k$-many dots on the left side. By Proposition \ref{prop: orientedavoidant}, the number of such oriented avoidant circles in the sphere is 
\[
2(2k)k -2k^2 - 2k + 2
= 2k^2 - 2k + 2
\]
Dividing by $2$, there are $k^2 - k + 1$ unoriented avoidant circles with $k$-many dots on each side.
\end{proof}

\section{Moving the dots}

\def\LConf{\operatorname{LConf}}
\def\Conf{\operatorname{Conf}}

We would like to understand how the preceding constructions (incident circles, avoidant circles, and higher order Voronoi decompositions) vary as we continuously deform the configuration of dots.
We do this by considering the \textbf{configuration space of $n$-many labeled dots in the sphere $S$}, which is the following topological space:
\[
\LConf_n(S) := \{ (d_1,d_2,\dotsc,d_n) \in S\times S\times \dotsb \times S \text{ such that, for all $i\neq j$, $d_i\neq d_j$}\}
\]
as well as the \textbf{configuration space of $n$-many unlabeled dots in the sphere $S$}, which is the quotient of $\LConf_n(S)$ by the action of the symmetric group $\Sigma_n$ by permuting the labels:
\[
\Conf_n(S) := \LConf_n(S) / \Sigma_n
\]
Since the symmetric group is finite and acts freely, the quotient map 
\[
\LConf_n(S)\rightarrow \Conf_n(S)
\]
is an $n!$-sheeted covering map. Since $S$ is a 2-dimensional manifold, $\LConf_n(S)$ and $\Conf_n(S)$ are $2n$-dimensional manifolds.
%
%

A configuration $D$ of $n$-many (unlabeled) dots in $S$ can be regarded as a point in $\Conf_n(S)$, and each way of indexing those points by the numbers $1,2,\dotsc,n$ gives a point in $\LConf_n(S)$.
Define
\[
\LConf_n^g(S)
\subset
\LConf_n(S)
\text{ and }
\Conf_n^g(S)
\subset
\Conf_n(S)
\]
to be the subsets consisting of configurations of (labeled or unlabeled) dots in general position. 

Unfortunately, nothing interesting will happen if we vary a configuration of dots inside $\Conf_n^g(S)$. The incident circles, avoidant circles, and higher order Voronoi decompositions will deform but not undergo any qualitative changes. 
To get a more interesting answer, we have to consider what happens as a configuration `crosses a wall' in $\Conf_n(S)$ to a different component of $\Conf_n^g(S)$.


%
These \emph{walls} are easiest to describe in the labeled configuration space $\LConf_n(S)$. For each $4$-element subset $I\subset \{1,2,\dotsc,n\}$, let
$
W_I\subset \LConf_n(S)
$
denote the subset in which the four dots labeled by $I$ are cocircular. The dimensions of these subsets and their intersections are as follows.

\begin{lemma}
Let $n$ be a positive integer.
\begin{enumerate}
    \item For each quadruple $I\subset \{1,2,\dotsc,n\}$, $W_I$ is a submanifold of dimension $2n-1$ in $\LConf_n(S)$.
    \item For any two distinct quadruples $I,J\subset \{1,2,\dotsc,n\}$, the intersection $W_I\cap W_J$ is a submanifold of dimension $2n-2$ in $\LConf_n(S)$.
\end{enumerate}
\end{lemma}


\begin{proof}
Permuting indices as needed, we may assume $n\in I$ and consider the map
\[
\LConf_n(S)\rightarrow \LConf_{n-1}(S)
\]
which deletes the $n$th dot $d_n$. The fiber over a labeled configuration $(d_1,d_2,\dotsc,d_{n-1})\in \LConf_{n-1}(S)$ is the space of choices of $d_n$, which may be identified with $S\smallsetminus \{d_1,d_2,\dotsc,d_{n-1}\}$.

If we let $i,j,k$ be the other three elements in $I$, then the restriction of $W_I$ to this fiber consists of points in $S\smallsetminus \{d_1,d_2,\dotsc,d_{n-1}\}$ on the circle through $d_i,d_j,d_k$; this consists of disjoint open intervals, which are submanifolds of dimension $1$. Varying the configuration $(d_1,d_2,\dotsc,d_{n-1})\in \LConf_{n-1}(S)$ continuously varies these intervals, and so
$
W_I\rightarrow \LConf_{n-1}(S)
$ is a bundle of dimension $1$ over $\LConf_{n-1}(S)$. Since $\LConf_{n-1}(S)$ has dimension $2n-2$, the dimension of $W_I$ is $2n-1$.

Let $J$ be a distinct quadruple of indices from $I$; permuting indices as needed, we may assume $n\in I$ and $n\not\in J$. Letting $W_J\subset \LConf_n(S)$ and $W_J'\subset \LConf_{n-1}(S)$ be the subsets where the dots indexed by $J$ are cocircular, the restriction of the map $\LConf_n(S)\rightarrow \LConf_{n-1}(S)$ to $W_J'\subset \LConf_{n-1}(S)$ is the map $W_J\rightarrow W_J'$, and the restriction of $W_I\subset \LConf_n(S)$ is the map $W_I\cap W_J\rightarrow W_J'$. Therefore, $W_I\cap W_J\rightarrow W_J'$ is a bundle of dimension $1$ over $W_J'$. By the previous part, $W_J'$ has dimension $2n-3$, so the dimension of $W_I\cap W_J$ is $2n-2$.
%
%
\end{proof}

The lemma implies that double intersections like $W_I\cap W_J$ are too small to disconnect $\LConf_n(S)$, and therefore any two points in $\LConf_n(S)$ can be connected by a path that crosses one wall at a time. This can be rephrased and strengthened as follows.

\begin{prop}\label{prop: semigeneral}
Any two configurations $D,D'$ of $n$-many (labeled or unlabeled) dots in general position can be connected by a family $D(t)$ of configurations parametrized by $t$, such that
\begin{enumerate}
    \item for all but finitely many $t$, $D(t)$ is in general position, and
    \item for all $t$, $D(t)$ is in \textbf{semigeneral position}; that is, at most one quadruple of dots is cocircular.
\end{enumerate}
\end{prop}

Such a family $D(t)$ will be called a \textbf{semigeneral family} from $D$ to $D'$. 

\begin{proof}
Consider the case when $D,D'$ are labeled; that is, $D,D'\in \LConf_n^g(S)$.
%
Consider the subset
\[ \LConf_n^{sg}(S) : = \LConf_n(S) \smallsetminus \bigcup (W_I\cap W_J) \]
where the union runs over all distinct quadruples of indices $I$ and $J$. Points in $\LConf_n^{sg}(S)$ correspond to labeled configurations in semigeneral position.
Since $\LConf_n(S)$ is path-connected and each $W_I\cap W_J$ has codimension $2$ inside $\LConf_n(S)$, $\LConf_n^{sg}(S)$ is also path-connected. In particular, we may choose a path ${D}(t)$ from ${D}$ to ${D}'$ in $\LConf_n^{sg}(S)$ parametrized by $t$. 

For each quadruple $I$, we consider the intersections of $W_I$ with ${D}(t)$. By deforming ${D}(t)$ in a tubular neighborhood of $W_I$, we may assume that the set of intersections is discrete (i.e.~without accumulation points). Since the path ${D}(t)$ is a compact set, this implies the set of intersections with $W_I$ is finite. Repeating for each $I$, we conclude that ${D}(t)$ has finitely many intersection points with the union of the $W_I$; equivalently, finitely many points where it is not in general position.

In the case of unlabeled $D$ and $D'$, choose arbitrary labelings and a semigeneral family between them in $\LConf_n^{sg}(S)$. The image of this path in $\Conf_n(S)$ is a semigeneral family from $D$ to $D'$.
%
\end{proof}

\section{Relations between higher order Voronoi decompositions}\label{sec: Relations}

To understand how higher order Voronoi decompositions change as we vary the configuration of dots, the proposition says that it suffices to understand how they change as a family of configurations crosses a single semigeneral wall. 

Recall from Section \ref{sec: higher Voronoi} that the vertices of the $k$th order Voronoi decomposition can be colored according to the number of dots on the left of the corresponding incident circle. 

\begin{lemma}\label{lemma: move}
Let $0<k<n>3$, let $D,D'\in \Conf^g_n(S)$, and 
Let $D(t)$ be a semigeneral family from $D$ to $D'$ which is general except at one value of $t$. Then the $k$th order Voronoi decompositions of $D$ and $D'$ are related by either
\begin{enumerate}
    \item one of the moves in Figure \ref{fig: postnikovmoves}, or
    \item two of the moves in Figure \ref{fig: postnikovmoves}, applied in a pair of antipodal circles.
\end{enumerate}


\end{lemma}

\begin{proof}
Consider a semigeneral family $D(t)$ which is general except at $t=0$. Choose an indexing of the dots so that $d_1(0),d_2(0),d_3(0),d_4(0)$ lie on a common circle, oriented so that it passes through the dots in that order. Let $C$ denote this oriented circle, and let $a$ be its left center.





For each $t$, let $C_1(t)$ denote the oriented incident circle through $d_2(t),d_3(t),d_4(t)$ (in that circular order), and let $a_1(t)$ denote the left center of $C_1(t)$. Similarly, let $C_2(t), C_3(t),C_4(t)$ denote the oriented incident circles through other three triples, and let $a_2(t),a_3(t),a_4(t)$ be their left centers. At $t=0$, these circles and their left centers all coincide, but they are distinct for other $t$. 

Therefore, we may choose a sufficiently small $\epsilon>0$ such that, for all $t\in [-\epsilon,\epsilon]$, the left centers $a_1(t),a_2(t),a_3(t),a_4(t)$ stay in an arbitrarily small circular neighborhood $U$ of $a$ (see Figure \ref{fig: crossingawall}).

\begin{figure}[h!t]
\[
\begin{tikzpicture}[scale=.5]
    \begin{scope}
        \draw[dashed] (0,0) circle (3.25);
        \node[dot] (1) at (180:1) {};
        \node[left] at (1) {$a_1(t)$};
        \node[dot] (2) at (90:1) {};
        \node[above] at (2) {$a_2(t)$};
        \node[dot] (3) at (0:1) {};
        \node[right] at (3) {$a_3(t)$};
        \node[dot] (4) at (270:1) {};
        \node[below] at (4) {$a_4(t)$};
        \node at (0,-4) {Shortly before $t=0$};
    \end{scope}
    \begin{scope}[xshift=8cm]
        \draw[dashed] (0,0) circle (3.25);
        \node[dot] (1) at (0,0) {};
        \node[above] at (1) {$a$};
        \node at (0,-4) {$t=0$};
    \end{scope}
    \begin{scope}[xshift=16cm]
        \draw[dashed] (0,0) circle (3.25);
        \node[dot] (1) at (180:1) {};
        \node[left] at (1) {$a_3(t)$};
        \node[dot] (2) at (90:1) {};
        \node[above] at (2) {$a_4(t)$};
        \node[dot] (3) at (0:1) {};
        \node[right] at (3) {$a_1(t)$};
        \node[dot] (4) at (270:1) {};
        \node[below] at (4) {$a_2(t)$};
        \node at (0,-4) {Shortly after $t=0$};
    \end{scope}
\end{tikzpicture}
\]
\caption{Left centers in a small neighborhood $U$ of $a$}
\label{fig: crossingawall}
\end{figure}
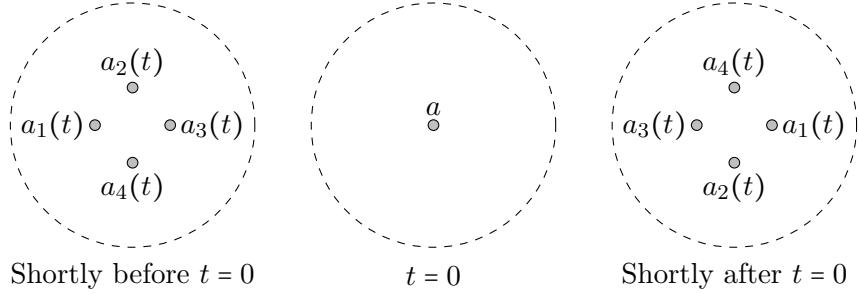

By Proposition \ref{prop: Voronoivertex}, the left centers $a_i(t)$ are vertices of the $k$th order Voronoi decomposition if and only if $C_i(t)$ has $k-2$ or $k-1$ many dots on its left side.
Let $D_-$ denote the dots which are on the left side of $C$; for small enough $\epsilon$, these dots are also on the left side of $C_i(t)$ for all $i\in \{1,2,3,4\}$ and $t\in [-\epsilon,\epsilon]$.

\begin{itemize}
    \item Consider $t\in [-\epsilon,\epsilon]$ for which the dot $d_1(t)$ is on the left side of $C_1(t)$. Then $d_2(t)$ is on the right side of $C_2(t)$, $d_3(t)$ is on the left side of $C_3(t)$, and $d_4(t)$ is on the right side of $C_4(t)$.
    
    The only way for any of the $a_i(t)$ to be vertices of the $k$th order Voronoi decomposition is if $|D_-|$ is between $k-3$ and $k-1$, in which case the restriction to the neighborhood $U$ will be equivalent to one of the pictures in Figure \ref{fig: d1left}.
    

\begin{figure}[h!t]
\[
\begin{tikzpicture}[scale=.5]
    \begin{scope}
        \node at (0,-4) {$|D_-|=k-3$};
        \draw[dashed] (0,0) circle (3.25);
        \clip (0,0) circle (3.25);
        \node[dot,thick,fill=white] (1) at (180:1) {};
        \node[left] at (1) {$a_1(t)$};
        \node[dot,thick,fill=white] (3) at (0:1) {};
        \node[right] at (3) {$a_3(t)$};

        \draw[thick] (1) to (3);
        \draw[thick] (1) to (-5,5);
        \draw[thick] (1) to (-5,-5);
        \draw[thick] (3) to (5,5);
        \draw[thick] (3) to (5,-5);
    \end{scope}
    \begin{scope}[xshift=8cm]
        \node at (0,-4) {$|D_-|=k-2$};
        \draw[dashed] (0,0) circle (3.25);
        \clip (0,0) circle (3.25);
        \node[dot,thick,fill=black] (1) at (180:1) {};
        \node[above left] at (1) {$a_1(t)$};
        \node[dot,thick,fill=white] (2) at (90:1) {};
        \node[above right] at (2) {$a_2(t)$};
        \node[dot,thick,fill=black] (3) at (0:1) {};
        \node[below right] at (3) {$a_3(t)$};
        \node[dot,thick,fill=white] (4) at (270:1) {};
        \node[below left] at (4) {$a_4(t)$};

        \draw[thick] (1) to (2);
        \draw[thick] (2) to (3);
        \draw[thick] (3) to (4);
        \draw[thick] (4) to (1);
        \draw[thick] (1) to (-5,0);
        \draw[thick] (2) to (0,5);
        \draw[thick] (3) to (5,0);
        \draw[thick] (4) to (0,-5);
    \end{scope}
    \begin{scope}[xshift=16cm]
        \node at (0,-4) {$|D_-|=k-1$};
        \draw[dashed] (0,0) circle (3.25);
        \clip (0,0) circle (3.25);
        \node[dot,thick,fill=black] (2) at (90:1) {};
        \node[ right] at (2) {$a_2(t)$};
        \node[dot,thick,fill=black] (4) at (270:1) {};
        \node[ left] at (4) {$a_4(t)$};

        \draw[thick] (2) to (4);
        \draw[thick] (2) to (-5,5);
        \draw[thick] (4) to (-5,-5);
        \draw[thick] (2) to (5,5);
        \draw[thick] (4) to (5,-5);
    \end{scope}
\end{tikzpicture}
\]
\caption{Local structure of $k$th order Voronoi decomp, when $d_1(t)$ is left of $C_1(t)$}
\label{fig: d1left}
\end{figure}
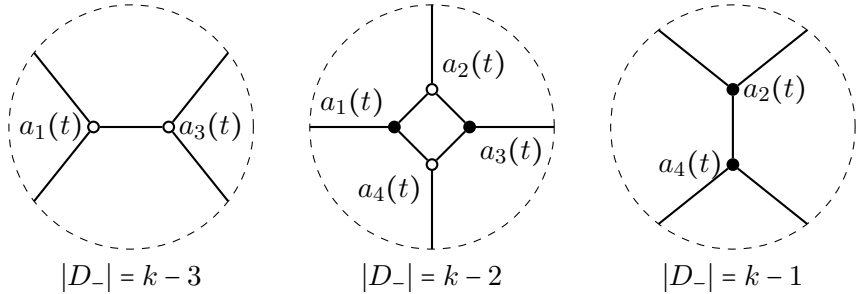
    \item Consider $t\in [-\epsilon,\epsilon]$ for which the dot $d_1(t)$ is on the right side of $C_1(t)$. Then $d_2(t)$ is on the left side of $C_2(t)$, $d_3(t)$ is on the right side of $C_3(t)$, and $d_4(t)$ is on the left side of $C_4(t)$.
    
    The only way for any of the $a_i(t)$ to be vertices of the $k$th order Voronoi decomposition is if $|D_-|$ is between $k-3$ and $k-1$, in which case the restriction to the neighborhood $U$ will be equivalent to one of the pictures in Figure \ref{fig: d1right}.
    

\begin{figure}[h!t]
\[
\begin{tikzpicture}[scale=.5]
    \begin{scope}
        \node at (0,-4) {$|D_-|=k-3$};
        \draw[dashed] (0,0) circle (3.25);
        \clip (0,0) circle (3.25);
        \node[dot,thick,fill=white] (2) at (90:1) {};
        \node[ right] at (2) {$a_2(t)$};
        \node[dot,thick,fill=white] (4) at (270:1) {};
        \node[ left] at (4) {$a_4(t)$};

        \draw[thick] (2) to (4);
        \draw[thick] (2) to (-5,5);
        \draw[thick] (4) to (-5,-5);
        \draw[thick] (2) to (5,5);
        \draw[thick] (4) to (5,-5);
    \end{scope}
    \begin{scope}[xshift=8cm]
        \node at (0,-4) {$|D_-|=k-2$};
        \draw[dashed] (0,0) circle (3.25);
        \clip (0,0) circle (3.25);
        \node[dot,thick,fill=white] (1) at (180:1) {};
        \node[above left] at (1) {$a_1(t)$};
        \node[dot,thick,fill=black] (2) at (90:1) {};
        \node[above right] at (2) {$a_2(t)$};
        \node[dot,thick,fill=white] (3) at (0:1) {};
        \node[below right] at (3) {$a_3(t)$};
        \node[dot,thick,fill=black] (4) at (270:1) {};
        \node[below left] at (4) {$a_4(t)$};

        \draw[thick] (1) to (2);
        \draw[thick] (2) to (3);
        \draw[thick] (3) to (4);
        \draw[thick] (4) to (1);
        \draw[thick] (1) to (-5,0);
        \draw[thick] (2) to (0,5);
        \draw[thick] (3) to (5,0);
        \draw[thick] (4) to (0,-5);
    \end{scope}
    \begin{scope}[xshift=16cm]
        \node at (0,-4) {$|D_-|=k-1$};
        \draw[dashed] (0,0) circle (3.25);
        \clip (0,0) circle (3.25);
        \node[dot,thick,fill=black] (1) at (180:1) {};
        \node[left] at (1) {$a_1(t)$};
        \node[dot,thick,fill=black] (3) at (0:1) {};
        \node[right] at (3) {$a_3(t)$};

        \draw[thick] (1) to (3);
        \draw[thick] (1) to (-5,5);
        \draw[thick] (1) to (-5,-5);
        \draw[thick] (3) to (5,5);
        \draw[thick] (3) to (5,-5);
    \end{scope}
\end{tikzpicture}
\]
\caption{Local structure of $k$th order Voronoi decomp, when $d_1(t)$ is right of $C_1(t)$}
\label{fig: d1right}
\end{figure}
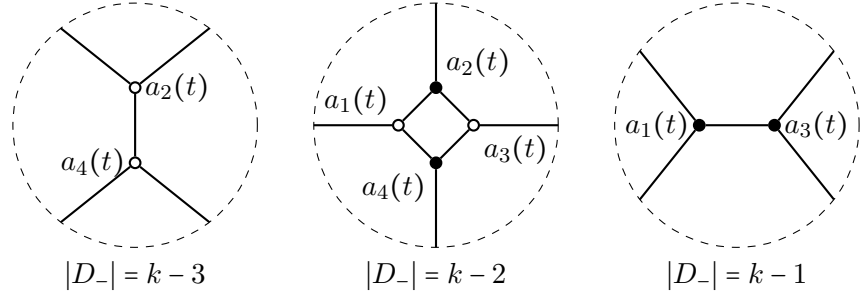
    
\end{itemize}

If $d_1(t)$ switches sides of $C_1(t)$ as $t$ crosses $0$, then the $k$th order Voronoi decomposition changes by one of the local moves in Figure \ref{fig: postnikovmoves}. If $d_1(t)$ stays on the same side of $C_1(t)$ as $t$ crosses $0$, then the $k$th order Voronoi decompositions before and after are equivalent.


Next, let $\overline{C}_i(t)$ denote the orientation-reversal of $C_i(t)$ for each $i$. The left center of $\overline{C}_i(t)$ is the right center of $C_i(t)$, which is the antipode $\sigma(a_i(t))$ of the left center $a_i(t)$. Just like Figure \ref{fig: crossingawall}, the four antipodes $\sigma(a_1(p)),\sigma(a_1(p)),\sigma(a_1(p)),\sigma(a_1(p))$ come together to single point $\sigma(a)$ before splitting again, and therefore the Voronoi decomposition changes by a local move or remains equivalent.

For any triple of dots besides the four triples in $\{d_1,d_2,d_3,d_4\}$, the corresponding incident circle does not cross a dot, and therefore has the same set of dots on either side for all $t$. As a consequence, the corresponding vertices in the $k$th order Voronoi decomposition and their adjacencies to other vertices do not change as $t$ varies. Therefore, the $k$th order Voronoi decomposition remains equivalent outside the small neighborhoods of $a$ and $\sigma(a)$.
\end{proof}

\begin{rem}
The proof demonstrates that moves occur in pairs only when $k$ is close to $\frac{1}{2}n$. When $k=\frac{1}{2}n$, then every move must occur in antipodal pairs (to preserve the symmetry from Proposition \ref{prop: antipodal}). When $n=2k\pm1$, then only some moves will occur in pairs; specifically, every square move will be accompanied by one of the merging-splitting moves.
\end{rem}

Proposition \ref{prop: semigeneral} and Lemma \ref{lemma: move} immediately imply the following.

\begin{rethm}{\ref{thm: move equiv}}
Let $0<k<n>3$. The $k$th order Voronoi decompositions of any two sets of $n$-many dots in general position in the sphere may be related by a sequence of the local moves in Figure \ref{fig: postnikovmoves}.
\end{rethm}

The moves in Figure \ref{fig: postnikovmoves} between bicolored graphs appeared in the seminal work of Postnikov, where such moves induced an equivalence between many associated constructions, such as \emph{boundary measurement maps}, \emph{trip permutations}, and \emph{cluster algebras}. While the first two constructions do not have an obvious generalization to bicolored graphs in the sphere, we briefly consider the associated cluster algebra in Section \ref{sec: future}. 

\begin{rem}
Theorem \ref{thm: move equiv} provides a second proof that the numbers of black vertices, white vertices, edges, and regions in a spherical $k$th order Voronoi decomposition only depends on $k$ and $n$, since the local moves preserve these counts. While it doesn't provide any counts, one could pick a particular nice configuration and count strata there to (re)deduce the formulas in Theorem \ref{thm: stratacount}. A proof in this spirit of the $k=\ell$ case of  Theorem \ref{thm: incident} is given in \cite{Ard04}.
\end{rem}

\begin{warn}
We caution the reader that some Postnikov moves cannot be realized by a semigeneral family! The most obvious obstruction is that, when $k=\frac{1}{2}n$, the $k$th order Voronoi decomposition is antipodally symmetric, and so semigeneral homotopies \emph{must} induce local moves in antipodal pairs. However, one may construct counterexamples for other values of $k$ as well.
\end{warn}

\section{Future directions}\label{sec: future}

We close by sketching some of the further directions one may take these ideas in, particularly those directions opened up by Theorem \ref{thm: move equiv}.

In \cite{Pos}, the author shows how a bicolored graph in a surface can be used to define a \emph{cluster algebra}, a special kind of commutative ring with distinguished elements called \emph{cluster variables} (a more modern exposition of this construction can be found in \cite{FWZdd}).
Given a configuration $D$ of $n$-many dots in general position in the sphere $S$, 
let $\mathcal{A}_\circlearrowleft(k,D)$ denote the cluster algebra associated to the $k$th order Voronoi decomposition of $S$ determined by $D$.

Since the moves in Figure \ref{fig: postnikovmoves} each induce an isomorphism between the associated cluster algebras \cite[Proposition 7.2.2]{FWZdd}, Theorem \ref{thm: move equiv} immediately implies the following.

\begin{prop}
Let $0<k<n>3$, and let $D, D'$ be configurations of $n$-many dots in general position in $S$.
Then a semigeneral family from $D$ to $D'$ induces an isomorphism of cluster algebras
\[
\mathcal{A}_\circlearrowleft(k,D) \rightarrow \mathcal{A}_\circlearrowleft(k,D')
\]
In particular, the isomorphism class of $\mathcal{A}_\circlearrowleft(k,D)$ only depends on $k$ and $n$.
\end{prop}

As we hope to show in a future work, this isomorphism only depends on the semigeneral family up to homotopy in $\Conf_n(S)$. As a consequence, elements in the fundamental group
\[
\pi_1(\Conf_n(S),D)
\]
induce well-defined automorphisms of $\mathcal{A}_\circlearrowleft(k,D)$. This fundamental group is called the \emph{$n$th spherical braid group}, and it can be identified with the \emph{mapping class group of the punctured sphere $S\smallsetminus D$}. Using this action, one can show that every homotopy class of simple, \emph{oriented} curve in $S\smallsetminus D$ with $k$-many dot on its left side determines a cluster variable in $\mathcal{A}_\circlearrowleft(k,D)$.


This suggests a natural variant for unoriented curves. Recall that, when $2k=|D|$, the $k$th order Voronoi decomposition has an antipodal symmetry. Let $\mathcal{A}_\bigcirc(k,D)$ be the cluster algebra of the quotient of this Voronoi decomposition by the antipodal symmetry\footnote{Or equivalently, as a \emph{fold} of the cluster algebra $\mathcal{A}_\circlearrowleft(k,D)$ by the antipodal symmetry (in the sense of \cite{Dup08}).}; for all other choices of $k$, let $\mathcal{A}_\bigcirc(k,D) :=\mathcal{A}_\circlearrowleft(k,D)$.  As before, the isomorphism class of $\mathcal{A}_\bigcirc(k,D)$ only depends on $k$ and $n:=|D|$, and one can show that every homotopy class of simple, \emph{unoriented} curve in $S\smallsetminus D$ with $k$-many dots on one of its sides determines a cluster variable in $\mathcal{A}_\bigcirc(k,D)$.

\begin{rem}
For small values of $n$, the cluster algebras $\mathcal{A}_\bigcirc(k,D)$ are of special and notable types. 

When $|D|=4$, the $\mathcal{A}_\bigcirc(2,D)$ is isomorphic to the \emph{Markov cluster algebra} \cite{BFZ05}, which is related to the Teichm\"uller space of the once-punctured torus and the theory of Markov triples. 

When $|D|=6$, the cluster algebra $\mathcal{A}_\bigcirc(3,D)$ is isomorphic to the \emph{$X_7$ cluster algebra}, an exotic and largely mysterious cluster algebra which appears as one of two sporadic cases in the classification of \emph{mutation-finite} cluster algebras \cite{DO08}. This realization (particularly, the action of the mapping class group) has non-trivial consequences for the structure of this cluster algebra, as well as for the combinatorics of separating curves in the closed surface of genus 2. We hope to explore these consequences in a future work.
\end{rem}


\appendix

\section{Connection between spherical and planar Voronoi decompositions}\label{appendix}

A curious feature of our proofs of Theorems \ref{thm: incident} and \ref{thm: avoidant} is they are both fundamentally spherical, even though the ultimate results are equally valid in the sphere and the plane.
One explanation is that our proofs boil down to counting various strata in the $k$th order Voronoi decomposition of the sphere (Theorem \ref{thm: stratacount}), and stereographic projection does \emph{not} send Voronoi decompositions to Voronoi decompositions. 
While stereographic projection sends circles to circles, it does not send the center of a circle to the center of its image.

However, there is a more subtle relationship between higher order Voronoi decompositions in the plane and sphere, which we sketch here for the benefit of the reader. In essence, the spherical $k$th order Voronoi decomposition can be constructed by gluing together the planar $k$th order Voronoi decomposition and the planar $(n-k)$th order Voronoi decomposition.

\begin{thm}[\cite{Brown79,NLC02,CHP22}]
Let $D$ be a configuration of dots in the sphere $S$, and let $\pi(D)$ be the stereographic projection in the plane $P$.
\begin{enumerate}
    \item There is a simple, piecewise-circular curve $R_0$ in $S$, disjoint from the vertices of the $k$th order Voronoi decomposition of $S$ determined by $D$, whose complement consists of two open regions $R_+$ and $R_-$.
    \item There is a homeomorphism from $R_+$ to the plane $P$ which induces a color-preserving equivalence between the $k$th order Voronoi decomposition of $S$ determined by $D$ (as restricted to $R_+$) and the $k$th order Voronoi decomposition of $P$ determined by $\pi(D)$.
    \item There is a homeomorphism from $R_-$ to the plane $P$ which induces a color-swapping equivalence between the $k$th order Voronoi decomposition of $S$ determined by $D$ (as restricted to $R_-$) and the $(n-k)$th order Voronoi decomposition of $P$ determined by $\pi(D)$.
\end{enumerate}
\end{thm}

The first version of this appeared in \cite{Brown79}, who proved parts (1) and (2) for $k=1$ to derive ($k=1$) Voronoi decompositions in the plane from (computationally faster) ($k=1$) Voronoi decompositions in the sphere. Part (3) of the $k=1$ case was shown in \cite{NLC02} with explicit maps. The general $k$ version appears in \cite{CHP22}, who used it to deduce the number of incident and avoidant circles from analogous results on planar Voronoi decompositions due to \cite{Lin03}.

\begin{rem}
The theorem implies that the number of white vertices in a spherical $k$th order Voronoi decomposition is equal to the number of white vertices in the corresponding planar $k$th order Voronoi decomposition plus the number of black vertices in the planar $(n-k)$th order Voronoi decomposition. Translating into circles, this is equivalent to a simple observation: given a configuration of $n$-many dots in general position in the plane, the number of incident circles with $(k-2)$-many dots on either side is equal to the number of incident circles with $(k-2)$-many dots in the interior plus the number of incident circles with $(n-k-1)$-many dots in the interior.
\end{rem}



\bibliographystyle{alpha}
\bibliography{circles}

\begin{thebibliography}{CAdlHPH22}

\bibitem[Ard04]{Ard04}
Federico Ardila.
\newblock The number of halving circles.
\newblock {\em Amer. Math. Monthly}, 111(7):586--591, 2004.

\bibitem[BFZ05]{BFZ05}
Arkady Berenstein, Sergey Fomin, and Andrei Zelevinsky.
\newblock Cluster algebras. {III}. {U}pper bounds and double {B}ruhat cells.
\newblock {\em Duke Math. J.}, 126(1):1--52, 2005.

\bibitem[Bro79]{Brown79}
Kevin~Q. Brown.
\newblock Voronoi diagrams from convex hulls.
\newblock {\em Information Processing Letters}, 9(5):223--228, 1979.

\bibitem[CAdlHPH22]{CHP22}
Merc{\`e} Claverol~Aguas, Andrea de~la Heras~Parilla, and Clemens Huemer.
\newblock Properties for {Voronoi} diagrams of arbitrary order on the sphere.
\newblock In {\em The 38th European Workshop on Computational Geometry}, pages
  210--217, 2022.

\bibitem[DO08]{DO08}
Harm Derksen and Theodore Owen.
\newblock New graphs of finite mutation type.
\newblock {\em Electron. J. Combin.}, 15(1):Research Paper 139, 15, 2008.

\bibitem[Dup08]{Dup08}
G.~Dupont.
\newblock An approach to non-simply laced cluster algebras.
\newblock {\em J. Algebra}, 320(4):1626--1661, 2008.

\bibitem[Ede87]{EdelsbrunnerBook}
Herbert Edelsbrunner.
\newblock {\em Algorithms in combinatorial geometry}, volume~10 of {\em EATCS
  Monographs on Theoretical Computer Science}.
\newblock Springer-Verlag, Berlin, 1987.

\bibitem[FWZ25]{FWZdd}
Sergey Fomin, Lauren Williams, and Andrei Zelevinsky.
\newblock Introduction to {Cluster} {Algebras}. {Chapter} 7, March 2025.
\newblock arXiv:2106.02160.

\bibitem[Lin03]{Lin03}
Roderik Lindenbergh.
\newblock A {V}oronoi poset.
\newblock {\em J. Geom. Graph.}, 7(1):41--52, 2003.

\bibitem[NLC02]{NLC02}
Hyeon-Suk Na, Chung-Nim Lee, and Otfried Cheong.
\newblock Voronoi diagrams on the sphere.
\newblock {\em Computational Geometry}, 23(2):183--194, 2002.

\bibitem[Pos06]{Pos}
Alexander Postnikov.
\newblock Total positivity, {G}rassmannians, and networks.
\newblock {\em preprint}, September 2006.
\newblock arXiv:math/0609764.

\end{thebibliography}

\end{document}